\documentclass[12pt,twoside]{amsart}
\usepackage{amscd,amssymb,epsfig,yhmath} 
\usepackage{wasysym}
\usepackage{hyperref}
\newcommand{\abs}[1]{\vert#1\vert}
\def\norm#1{\left\Vert#1\right\Vert}
\def\R {{\mathbb R}}

\def\I {{\mathbb I}}

\def\T {{\mathbb T}}
\def\N{{\mathbb N}}

\def\e{{\varepsilon}}

\def\s{{\mathbb S}}

\def\H{{\mathcal H}}

\def\Ad{{\mathrm{Ad}\,}}

\def\RUCB{{\mbox{\rm RUCB}\,}}
\def\LUCB{{\mbox{\rm LUCB}\,}}

\hypersetup{
  colorlinks=true,       
  linkcolor=magenta,          
    filecolor=magenta,      
    urlcolor=cyan           
}

\newtheorem{theorem}{Theorem}[section]

\newtheorem{lemma}[theorem]{Lemma}

\theoremstyle{definition}

\theoremstyle{remark}
\newtheorem{remark}[theorem]{Remark}

\numberwithin{equation}{section}

\begin{document}

\begin{center}
\renewcommand{\thefootnote}{\fnsymbol{footnote}}
{\large\bf Amenability of finite energy path and loop groups}
\\[.6cm]

{\sc Vladimir G. Pestov}
\\[.5cm]

{Departamento de Matem\'atica, Universidade Federal da Para\'\i ba, Jo\~ao Pessoa, PB, 58051-900 Brasil}
\\[1mm]
{\em and}
\\[1mm]
Departement of Mathematics and Statistics, University of Ottawa,
Ottawa, ON \\
K1N 6N5 Canada \footnote[5]{Emeritus Professor.}
\\[1mm]
vladimir.pestov@uottawa.ca
\\[.5cm]
\end{center}
\hrule
\vskip .4cm

{\footnotesize\bf Abstract}

\begin{quote}
{\footnotesize
\noindent
It is shown that the groups of finite energy (that is, Sobolev class $H^1$) paths and loops with values in a compact Lie group are amenable in the sense of Pierre de la Harpe, that is, every continuous action of such a group on a compact space admits an invariant regular Borel probability measure. To our knowledge, the strongest previously known result concerned the amenability of groups of continuous paths and loops (Malliavin and Malliavin 1992).}
\end{quote}
\vskip .4cm

\hrule 

\vskip 1cm

\section{Introduction}

Amenability of a locally compact group admits numerous equivalent definitions, see \cite{greenleaf,GdlH}. For more general topological groups they are not necessarily pairwise equivalent, and the property that is usually called amenability is this \cite{E}: a topological group $G$ is amenable if every continuous action of $G$ on a compact space $X$ admits an invariant regular Borel probability measure.
This is the concept we are interested in in this article. 

Carey and Grundling \cite{CG} discuss the existence of the gauge-invariant vacuum state for gauge field theories. As argued by the authors, in certain settings the existence of such a state can be derived from the amenability of the gauge transformation group.
The question asked in the above paper is therefore this: if $X$ is a compact Riemannian manifold, are the groups of currents $C^{\infty}(X,SU(n))$, $C^{k}(X,SU(n))$ and $H^k(X,SU(n))$, with their natural topologies, amenable? Of course the same can be asked about more general gauge transformation groups of vertical automorphisms of a principal $K$-fibre bundle, where $K$ is a compact connected Lie group. 

The question appears to be still open for all values of $k\geq 0$, with the only exception of groups of continuous loops $C(\s^1,K)$ and continuous paths $C(\I,K)$ with values in a compact Lie group $K$. Amenability of those groups follows from a result of Malliavin and Malliavin \cite{MM}, see an explanation in \cite{P20}, Sect. 2. Also note that if we reduce the class of smoothness even further and consider groups of measurable maps with the topology of convergence in measure, we get a very strong version of amenability called extreme amenability \cite{Gla98,PS1}.

Another source of interest in amenability of groups of maps with values in compact Lie groups (in the $C^0$ case) is the open problem of describing those $C^\ast$-algebras whose unitary groups are amenable in the norm topology. See \cite{AST,O} for the latest developments in this direction.

Here is the main result of this note.

\begin{theorem}
Let $K$ be a compact Lie group. The groups $H^1(\I,K)$ and $H^1(\s^1,K)$ of finite energy (that is, Sobolev class $H^1$) paths and loops, with the $H^1$-topology, are amenable. The same is true of their based versions, $H_0^1(\I,K)$ and $H_0^1(\s^1,K)$.
\label{th:main}
\end{theorem}

Amenability of a topological group $G$ can be reformulated in terms of invariant means as follows. Recall that a function $f$ on a topological group is called {\em right uniformly continuous} if for every $\e>0$ there is a neighbourhood of the identity, $V$, so that $\abs{f(x)-f(y)}<\e$ when $xy^{-1}\in V$. The collection of all bounded right uniformly continuous functions on $G$ (with the supremum norm) forms an abelian $C^\ast$-algebra, denoted $\RUCB(G)$. A topological group $G$ is amenable if and only if there exists a left-invariant mean on $\RUCB(G)$, that is, a positive linear functional of norm one invariant under left translations by elements of $G$. 

Denote $\LUCB(G)$ the space of bounded {\em left} uniformly continuous functions, that is, those satisfying: for every $\e>0$ there is a neighbourhood of the identity, $V$, so that $\abs{f(x)-f(y)}<\e$ when $x^{-1}y\in V$. 
The inverse element map $g\mapsto g^{-1}$ establishes an isometric isomorphism between the spaces of the left and the right uniformly continuous functions, and sends left-invariant means to the right-invariant ones. It follows immediately that a topological group $G$ is amenable if and only if it admits a {\em right-invariant} mean on the space of {\em left} uniformly continuous bounded functions, $\LUCB(G)$. We will use this reformulation in the proof of our main Theorem \ref{th:main}.

Now let us call a topological group {\em skew-amenable} if there exists a {\em left-invariant} mean on the algebra $\LUCB(G)$ of bounded left uniformly continuous functions. Even if formally the two definitions look very similar, the behaviour of amenable and skew-amenable groups differs significantly when $G$ is not locally compact. See \cite{P20,JS}. (For locally compact groups, skew-amenability is equivalent to amenability, see Th. 2.2.1 in \cite{greenleaf}.) It is interesting that skew-amenability also appears in the study of the problem of amenability of unitary groups of $C^\ast$-algebras \cite{AST,O}. 

It is easy to give examples of amenable groups that are not skew-amenable (such is the unitary group $U(\ell^2)$ with the strong topology, see \cite{P20}). The existence of skew-amenable non-amenable groups was only recently demonstrated by Ozawa, who has produced a large collection of very interesting naturally occurring examples, in particular the unitary groups of hyperfinite $II_1$-factors with their weak topology \cite{O}.

Given a unitary representation $\pi$ of a group $G$ in a Hilbert space $\H$, a state $\phi$ on the algebra $B(\H)$ (that is, a positive linear functional with $\phi(1)=1$) is {\em invariant} (under $\pi$) if $\phi(\pi_{g^\ast}T\pi_g)=\phi(T)$ for every bounded operator $T$ and each $g\in G$. A unitary representation admitting an invariant state is called amenable in the sense of Bekka \cite{Bek1}. It was observed by Thierry Giordano and this author in \cite{GP07}, Prop. 4.5 that every strongly continuous unitary representation of a skew-amenable topological group $G$ admits an invariant state.
In view of this, the following result by the present author 
\cite{P20} is of interest.

\begin{theorem}[\cite{P20}]
Let $K$ be a compact Lie group. The groups $H^1(\I,K)$ and $H^1(\s^1,K)$ of finite energy paths and loops are skew-amenable. The same is true of their based versions, $H_0^1(\I,K)$ and $H_0^1(\s^1,K)$.
\label{th:skew}
\end{theorem}

Theorems \ref{th:main} and \ref{th:skew} together imply the following.

\begin{theorem}
Let $K$ be a compact Lie group. The groups $H^1(\I,K)$ and $H^1(\s^1,K)$ of finite energy paths and loops admit a bi-invariant mean on the space of bounded right uniformly continuous functions. The same is true of their based versions, $H_0^1(\I,K)$ and $H_0^1(\s^1,K)$.
\label{th:bi}
\end{theorem}

Notice that the same holds for the groups $C(\I,K)$ and $C(\s^1,K)$ because they are SIN groups, that is, the left and the right uniform structures on them coincide, and so right and left uniformly continuous functions are the same.

In the above results, it is not even necessary to assume that $K$ is connected. The result for an arbitrary compact Lie group follows in a somewhat trivial way, because then the groups of paths and loops are just extensions of the corresponding groups with values in the connected component of $K$ by a finite subgroup. Such an extension preserves amenability, skew-amenability, the existence of required invariant means etc., so for the proofs, we will always assume $K$ to be connected, as this is the only interesting case.

The paper by Malliavin and Malliavin \cite{MM} had in fact established a stronger property of the groups $C(\I,K)$ and $C(\s^1,K)$ than mere amenability: they admit a mean on the space of all bounded Borel measurable functions that is bi-invariant under the action of the subgroup of all $C^1$-paths (resp., $C^1$-loops). We do not know if the groups $H^1(\I,K)$ and $H^1(\s^1,K)$ admit a mean on the bounded Borel functions that is bi-invariant (or: just invariant on the left) under the action of $H^2$-paths/loops.

For a more detailed discussion of those concepts of amenability as well as more references on the subject, see \cite{PS2}.

The preliminaries are collected in Section \ref{s:prelim}, while
Section \ref{s:asym} contains the main technical construction, that of a sequence of probability measures on the group of based paths that are asymptotically invariant with regard to the mass transportation distance. From there, we deduce the amenability of path groups in the short Section \ref{s:proofs}. Amenability of loop groups, proved in Section \ref{s:loop}, requires the following, apparently new, observation (Theorem \ref{th:coco}): a co-compact normal subgroup of an amenable Polish group is amenable. 

\section{Preliminaries}
\label{s:prelim}

Here we will present a summary of necessary concepts and results related to the groups of finite energy paths and loops, referring the reader to a somewhat more detailed presentation with references in \cite{P20}, Sect. 3.

Let $K$ be a compact connected Lie group, and $f\colon \I\to K$ a function. The value of the {\em right logarithmic derivative} of $f$ at a point $t\in\I$ of smoothness of $f$ is an element of the Lie algebra $\mathfrak k$ of $K$, given by
\begin{equation}
\partial^{log}f(t) = f^{\prime}(t)\cdot f(t)^{-1},
\label{eq:rightlog}
\end{equation}
the image of the derivative $f^{\prime}(t)\in T_{f(t)}K$ under the right translation by $f(t)^{-1}$.
We will always assume that $K$ is realized as a matrix subgroup of the orthogonal group, in which case on the right we have a product of matrices.

Now equip the Lie algebra $\mathfrak k$ with an $\Ad$-invariant inner product.
An absolutely continuous mapping $f\colon \I\to K$ has the right logarithmic derivative defined almost everywhere, and is said to have {\em finite energy} if
\[\int_0^1 \norm{\partial^{log}f(t)}^2dt<\infty.\]
In the matrix case, it is equivalent to saying that $f$ belongs to the Sobolev class $H^1$. The family of all maps $\I\to K$ of finite energy forms a group under the pointwise multiplication, the finite energy path group, denoted $H^1(\I,K)$. 

The right logarithmic derivative maps $H^1(\I,K)$ to $L^2(\I,{\mathfrak k})$. For any two elements $f,g\in H^1(\I,K)$, a direct calculation verifies the cocycle property: for a.e. $t\in\I$,
\begin{equation}
\partial^{log}(fg)(t)=\partial^{log}f(t) + \Ad_{f(t)}\partial^{log}g(t).
\label{eq:cocycle}
\end{equation}
In the case where $K$ is a matrix group, the adjoint representation is given by conjugation $\Ad_fg=fgf^{-1}$.

The right logarithmic derivative becomes injective if restricted to the subgroup $H_0^1(\I,K)$ of based paths, defined by the property $f(0)=e$. In this case, it is even a bijection, with the inverse map given by the product integral $\prod_0^tf(s)\,ds$. For a step function $f$ assuming constant values in $\mathfrak k$ on each interval $[t_i,t_{i+1})$, $i=0,1,\ldots,n-1$, the product integral of $f$ between $0$ and $t\in [t_j,t_{j+1})$ is defined by
\begin{equation}
\prod_0^t\exp f(s)\,ds =\exp_G (t-t_j)f(t_j)\exp_G(t_j-t_{j-1})f(t_{j-1})\ldots\exp_G t_1f(0).
\label{eq:prod}
\end{equation}
Afterwards, the product integral is extended over all $L^1$-functions by continuity. In particular, it is well defined on the $L^2$ functions. (See \cite{DF}, in particular Thm. 1.2, and Sect. 1.8, p. 55, eq. (8.6).)

 The formula
\[d(f,g)=\norm{\partial^{log}f-\partial^{log}g}_2\]
defines a metric on $H_0^1(\I,K)$. Eq. (\ref{eq:cocycle}) implies that this metric is left-invariant: $d(hf,hg)=d(f,g)$. One can verify that it turns $H_0^1(\I,K)$ into a topological group. Since the metric is complete and separable, $H_0^1(\I,K)$ is a Polish group, in fact, it is also a Banach--Lie group (\cite{AHKMT}, Sect. 1.8). The bijection given by the right logarithmic derivative and the product integral allows to identify $H_0^1(\I,K)$, as a metric group, with the Hilbert space $L^2(\I,{\mathfrak k})$ equipped with the group law
\begin{equation}
f\ast g =f+\Ad_{\Pi\exp f}g.
\label{eq:grouplaw}
\end{equation}
Here $\Pi\exp f$ is the short for the function $t\mapsto \prod_{0}^t \exp f(s)\,ds$, and the adjoint representation $\Ad$ is applied pointwise. 

In particular, zero is the neutral element of this group law, while the inverse to a function $f$ is given by
\begin{equation}
f^{\ast-1}=-\Ad_{\left(\prod\exp f\right)^{-1}}f.
\label{eq:inverse}
\end{equation}
(To check the claim, it is enough to work out just the product $f\ast f^{\ast-1}$.)

Moving on to the larger group $H^1(\I,K)$, the natural topology on it can be now defined, for example, by the left-invariant metric
\[d(f,g)=\norm{\partial^{log}f-\partial^{log}g}_2 + \rho(f(0),g(0)),\]
where $\rho$ is any left-invariant metric on $K$. This turns $H^1(\I,K)$ into a Polish group, containing $H_0^1(\I,K)$ as a co-compact normal subgroup.

The group $H_0^1(\s^1,K)$ of based finite energy loops is a Polish group isomorphic to a normal cocompact subgroup of $H_0^1(\I,G)$ given by the condition $f(1)=e$. Free finite energy loops form a Polish group, $H^1(\s^1,K)$, which is in its turn isomorphic to a closed (non-normal!) co-compact subgroup of $H^1(\I,K)$, given by the condition $f(0)=f(1)$. The group $H^1(\s^1,K)$ contains $H_0^1(\s^1,K)$ as a normal co-compact subgroup.

Let us also remark that if $K$ is non-abelian, then on all four groups of paths/loops above the left and right uniform structures are different (\cite{P20}, Prop. 7), which implies \cite{protasov} (given that our groups are metrisable, as the general case remains unsettled \cite{itzkowitz}) that the spaces of right and left uniformly continuous functions on them are different as well.

\section{The construction of asymptotically invariant measures}
\label{s:asym}

In this section we will mostly work with an isomorphic copy of the topological group $H^1_0(\I,K)$ which is the Hilbert space $L^2(\I,{\mathfrak k})$ equipped with the group operation $\ast$ as in Eq. (\ref{eq:grouplaw}).

For every $N$, denote $P_N$ the uniform partition of the unit interval $\I$ into subintervals of length $1/N$. Let $V_N$ be the vector space of all functions in $L^2(\I,{\mathfrak k})$ constant a.e. on elements of $P_N$. The union
\[V_{\infty}=\bigcup_{N=1}^{\infty} V_N\]
forms a vector space of simple functions which is dense in $L^2((0,1),{\mathfrak k})$.

Denote $B_N$ the unit ball in the $Nd$-dimensional Euclidean space $V_N$, where $d=\dim{\mathfrak k}$. Given $R>0$, let $\nu_{N,R}$ denote the restriction of the Lebesgue measure on the ball $RB_N$ of radius $R$ normalized so as to become a probability measure. 

In the next result, the Monge-Kantorovich (mass transportation) distance between probability measures is formed with regard to the metric on $L^2(\I,{\mathfrak k})$ that is the minimum of the norm distance and $1$. Clearly, this metric is also left-invariant under the multiplication law $\ast$ and generates the same  topology. For a comprehensive treatment of the mass transportation distance, we refer to \cite{villani}.

\begin{lemma}
Let $R_N=\omega(\sqrt N)\cap o(N)$.
Let $g\in V_{\infty}$ be a simple function. Then, as $N\to\infty$, the Monge-Kantorovich distance between the uniform measure $\nu_{N,R_N}$ on the ball $R_NB_N$ and its image $\nu_{N,R_N}\ast g$ under the right multiplication by $g$ converges to zero. Equivalently, for every $1$-Lipschitz function $F\colon L^2(\I,{\mathfrak k})\to \R$, satisfying $\norm{F}_{\infty}\leq 1$,
\[\int F(f)\,d(\nu_{N,R_N}) - \int F(f)\,d(\nu_{N,R_N}\ast g)\to 0
\mbox{ as }N\to\infty,\]
and the convergence is uniform over such $F$.
\label{l:main}
\end{lemma}

The rest of this Section is devoted to proving Lemma \ref{l:main}.

Consider any function $f\in L^2(\I,{\mathfrak k})$. Denote for simplicity 
\begin{equation}
r(t)=\Pi_0^t\exp f(s)\,ds.
\label{eq:r}
\end{equation}
This is an element of $H_0^1(\I,K)$. Given a uniform partition $0<t_1<\ldots<t_N=1$ of the unit interval, denote $\rho=\rho^f_N$ a step function from $\I$ to $K$ assuming the value $r(i/N)$ on the $i$-th interval of the partition, $[i/N,(i+1)/N)$. In other words,
\begin{equation}
\rho^f_N(t)=\Pi_0^{i/N}\exp f(s)\,ds,\mbox{ if }\frac iN\leq t<\frac{i+1}N.
\label{eq:rho}
\end{equation}

Now write
\begin{equation}
f\ast g = \underbrace{f + \Ad_{\rho^f_N}g}_{(I)} +\underbrace{\left(\Ad_{\Pi\exp f}g-\Ad_{\rho^f_N}g\right)}_{(II)}.
\label{eq:split}
\end{equation}
We will separately estimate the two expressions on the right. First, let us get out of the way the simpler estimate for $(II)$.

\begin{lemma}
The exponential map from $\mathfrak{so}(d)$ to $SO(d)$ is $1$-Lipschitz with regard to the uniform norm on $M_d(\R)$.
\label{l:explip}
\end{lemma}

For a sketch of an elegant and simple proof, see exercise 106 on the list of supplementary exercises to the book \cite{dserre}, pp. 75--76. (Note that the notation $\norm{\cdot}_2$ in item (b) of the exercise refers to the uniform operator norm, see subsection 4.1.2 in the main text of the book {\em ibid.})

\begin{remark}
The sequence of probability measures, whose asymptotic invariance with regard to the right multiplication we are proving, is the same as used by this author in \cite{P20} to prove the asymptotic invariance with regard to the left multiplication. The following lemma is more or less where the two arguments overlap. The estimates in the proof are mostly extracted from the proof of Lemma 10 in \cite{P20}. However, as the result was never isolated as such, and the present bounds are more exact, we have decided to both state it and equip with a complete proof.
\end{remark}

\begin{lemma}
Let $g\colon \I\to{\mathfrak k}$ be a simple function, taking the value $g_i$ on the $i$-th interval of the partition $P_N$. Denote $\norm{g}_{\infty}=\max_{i=1}^N\norm{g_i}_2$. Let $\rho_N^f$ be defined for every function $f$ as in Eq. (\ref{eq:rho}). 
Then,
\[\sup_{f\in R_NB_N}\left\Vert\Ad_{\Pi\exp f}g-\Ad_{\rho^f_N}g\right\Vert_2\leq 2\norm{g}_{\infty}\frac{R_N}{N}.\]
\label{l:I}
\end{lemma}

\begin{proof}
Let $f\in R_NB_N$. Denote $f_i\in{\mathfrak k}$ the contant value assumed by $f$ on the $i$-th interval of the uniform partition $P_N$. Thus, 
\[\frac 1{N}\sum_{i=0}^{N-1}\norm{f_i}_2^2 = \norm{f}_2^2\leq R^2_N.\] 
Here and in the sequel, we will also denote by $\norm{\cdot}_2$ the Hilbert--Schmidt norm, and $\norm{\cdot}_u$ the uniform norm on the matrices.
Remembering that the operator norm is bounded by the Hilbert--Schmidt norm and using Lemma \ref{l:explip}, given $t\in [i/N,(i+1)/N)$, we have
\begin{align*}
\norm{\rho(t)-r(t)}_u &=
\norm{r(i/N)-r(t)}_u \\ 
& = \norm{1-\exp{(t-i/N)f_i}}_u \\
&\leq (t-i/N)\norm{f_i}_2 \\
&\leq N^{-1}\norm{f_i}_2.
\end{align*}

Recall the formula $\norm{AB}_2\leq\norm{A}_2\norm{B}_u$.
For any two unitary matrices $u,v$ and a matrix $A\in M_n$, 
\begin{align*}
\norm{\Ad_uA-\Ad_vA}_2 &= \norm{uAu^{-1}-vAv^{-1}}_2 \\
& \leq \norm{uAu^{-1}-uAv^{-1}}_2 + \norm{uAv^{-1}-vAv^{-1}}_2 \\
&\leq 2\norm{u-v}_u\norm{A}_2.
\end{align*}
Now we have
\begin{align*}
\norm{ \Ad_rg-\Ad_{\rho} g}_2^2 &= \int_0^1 \norm{(\Ad_{r(x)}-\Ad_{\rho(x)})g(x)}_2^2d x\\
& \leq \sum_{i=0}^{N-1} \int_{i/N}^{(i+1)/N} 4\norm{r(x)-\rho(x)}_u^2\norm{g_i}_2^2d x\\
&\leq 4\sum_{i=0}^{N-1} \int_{i/N}^{(i+1)/N}N^{-2}\norm{f_i}_2^2  \max_{i=0}^{N-1}\norm{g_i}_2^2dx\\
&= 4N^{-2}\norm{g}^2_{\infty}\int_{0}^{1} \norm{f(x)}_2^2dx\\
&= 4N^{-2}\norm{g}^2_{\infty}\norm{f}^2_2.
\end{align*}
Thus,
\begin{align*}
\norm{\Ad_rg-\Ad_{\rho} g}_2&\leq 2N^{-1}\norm{g}_{\infty}\norm{f}_2 \\
&\leq 2\norm{g}_{\infty}\frac{R_N}{N}.
\label{eq:thus}
\end{align*}
\end{proof}

Now we concentrate on estimating $(I)$. 

\begin{lemma}
If $f$ is a step function taking constant values in $\mathfrak k$ on each interval $[t_i,t_{i+1})$, $i=0,1,\ldots,N-1$, then
\begin{equation}
\Ad_{(\Pi\exp f)^{-1}}f = \Ad_{{(\rho_N^f)}^{-1}}f.
\label{eq:adpexp=adrho}
\end{equation}
\label{l:adpexp=adrho}
\end{lemma}

\begin{proof}
According to the formula (\ref{eq:prod}) for the product integral of a step function, if $t\in [t_j,t_{j+1})$,
\begin{align}
\Ad_{(\Pi\exp f(t))^{-1}}f(t)&= \exp_K (-t_1f(0))\ldots \exp_K(-(t_j-t_{j-1})f(t_{j-1}))\exp_K (-(t-t_j)f(t_j))
f(t_j)\nonumber \\ &
\exp_K (t-t_j)f(t_j)\exp_K(t_j-t_{j-1})f(t_{j-1})\ldots\exp_K t_1f(0)\nonumber \\
&= \exp_K (-t_1f(0))\ldots \exp_K(-(t_j-t_{j-1})f(t_{j-1}))f(t_j)\label{eq:expansion}  \\ &\exp_K(t_j-t_{j-1})f(t_{j-1})\ldots\exp_K t_1f(0)\nonumber\\
&= \Ad_{{(\rho_N^f(t))}^{-1}}f(t).\nonumber 
\end{align}
Here we used the facts that on the $j$-th interval $f(t)\equiv f(t_j)$, $\rho_N^f(t)\equiv \rho_N^f(t_j)$, and the matrices $f(t_j)$ and $\exp(\pm (t-t_j) f(t_j))$ commute. 
\end{proof}

\begin{lemma}
The inverse map
\[f\mapsto f^{\ast-1}=-\Ad_{(\Pi\exp f)^{-1}}f\]
preserves the $dN$-dimensional Lebesgue measure on the Euclidean space $V_N$ and in particular preserves the measures $\nu_{N,R_N}$.
\label{l:f-1jacobian}
\end{lemma}

\begin{proof}
According to Lemma \ref{l:adpexp=adrho}, the linear subspace $V_N$ is closed under the inverse map, as indeed conjugation by the inverse of $\rho_N^f$ is an orthogonal transformation of $V_N$.
Since $\norm{f^{\ast-1}}_2=\norm{f}_2$, the inverse map is a non-linear bijection of each ball centred at zero onto itself. (Of course it is not an isometry, for otherwise it would be linear by Mazur--Ulam theorem.) Enough to verify the first claim. Let $f$ be a simple function, and denote again $f_i\in{\mathfrak k}$ the constant value of $f$ on the $i$-th interval. If in the formula (\ref{eq:expansion}) we 
treat the values $f_j=f(t_j)$ as the coordinate martix-valued variables in the space $V_N\cong {\mathfrak k}^N$, the map in the statement of our lemma can be written as a matrix function, say $\Phi$, whose $j$-th matrix component is
\begin{align*}
&\Phi_j(f_0,f_1,f_2,\ldots,f_{N-1})\\ 
& =-\exp_K (-t_1f_0)\ldots \exp_K(-(t_j-t_{j-1})f_{j-1})f_j\exp_K(t_j-t_{j-1})f_{j-1})\ldots\exp_K t_1f_0.
\end{align*}
The $j$-th component of the above function only depends on the coordinates up to and including the $j$-th.
The Jacobian matrix is therefore upper block-diagonal, whose diagonal is made of the (matrices representing the) differentials of the above functions, which are linear in $f_j$ when the values of $f_0,f_1,\ldots,f_{j-1}$ are fixed. The differential of a linear operator is the operator itself. The operator $x\mapsto -\Ad_ux$ is orthogonal and so has determinant $\pm 1$. Since $x\mapsto\Ad_ux$ is connected by a path with the identity operator, the determinant of the Jacobian is identically either plus or minus one.
\end{proof}

\begin{remark}
The above observation does not permit us to reduce the main theorem of this note to the argument used in \cite{P20}. The inverse map preserves the sequence of probability measures $\nu_{N,R_N}$ and replaces the right multiplication with the left one. So far, so good. However, it also replaces the left uniform structure (given by the familiar Hilbert norm) with the right uniform structure (for which we have no reasonable description). 
So overall, the proof of amenability turns out rather more involved than the proof of skew-amenability in \cite{P20}.
\end{remark}

\begin{lemma}
Let $g\in V_N$. The mapping 
\[V_N\ni f\mapsto f + \Ad_{\rho^f_N}g\in V_N\]
preserves the Lebesgue measure on $V_N$.
\label{l:id+const}
\end{lemma}

\begin{proof}
The argument is similar to the proof of Lemma \ref{l:f-1jacobian}. If we treat the mapping above as a function from ${\mathfrak k}^N$ to itself that depends on $N$ matrix coordinates and has $N$ matrix components, then the $j$-th component is of the form
\begin{align*}
&\Phi_j(f_0,f_1,f_2,\ldots,f_{N-1})\\ 
& =f_j+\exp_K (-t_1f_0)\ldots \exp_K(-(t_j-t_{j-1})f_{j-1})g_j\exp_K(t_j-t_{j-1})f_{j-1}\ldots\exp_K t_1f_0,
\end{align*}
where $g_j$ are constant matrices. Thus, the Jacobian matrix of the map has an upper triangular block-diagonal form. If we fix the coordinates $f_0,f_1,\ldots,f_{j-1}$, the block $(j,j)$ on the diagonal is the identity matrix, because it represents the differential of the map of the form ``identity $+$ constant.'' Thus, the determinant of the Jacobian is one.
\end{proof}

\begin{lemma}
Let $g\neq 0$ be fixed, and let $(R_N)$ be a sequence of strictly positive reals. The angle between a random element $f\in R_NB_N$, $f\sim\nu_{N,R_N}$ and $\Ad_{\rho_N^f}g$ concentrates exponentially around $\pi/2$, that is, for some suitable constants $C,c>0$ and all sufficiently small $\e>0$,
\[\nu_{N,R_N}\left\{f\in R_NB_N\colon \left\vert\angle\left(f,\Ad_{\rho_N^f}g\right)
-\frac{\pi}2\right\vert>\e\right\}<Ce^{-c\e^2N}.\]
\label{l:angleconc}
\end{lemma}

\begin{proof}
Without loss of generality we may suppose $\norm{g}_2=1$. Apply the Paul L\'evy concentration inequality for $1$-Lipschitz maps on the Euclidean ball of unit radius and dimension $dN=\Theta(N)$ (\cite{ledoux}, Prop. 2.9, p. 30) to the map $f\mapsto \langle f,g\rangle$. For every $\e>0$
\[\nu_N\left\{f\in B_N\colon \left\vert\langle f,g\rangle\right\vert>\e\right\}< 2\exp(-c\e^2N).\]
Also, for any fixed $\e>0$, the measure of the $\Theta(N)$-dimensional unit ball concentrates exponentially on the spherical $\e$-shell.
Consequently, 
\begin{align*}
& \nu_{N,R_N}\left\{f\in R_NB_N\colon \left\vert\langle f,g\rangle\right\vert>\norm f\cdot \e\right\} \\
&\leq
\nu_{N,R_N}\left\{f\in R_NB_N\colon \left\vert\langle f,g\rangle\right\vert>\frac 12 R_N\e\right\}+\nu_{N,R_N}\left(\frac{R_N}{2}B_N \right)\\
&=\nu_{N}\left\{f\in B_N\colon \left\vert\langle f,g\rangle\right\vert>\frac {\e}2 \right\}+\nu_{N}\left(\frac{1}{2}B_N \right)\\
& < C\exp(-c^\prime\e^2N).
\end{align*}

We have, using Lemma \ref{l:adpexp=adrho} and the unitarity of the adjoint representation,
\begin{align*}
\langle f,\Ad_{\rho_N^f}g\rangle &
= \langle (\Ad_{\rho_N^f})^{\ast}f,g\rangle \\
&= \langle \Ad_{(\rho_N^f)^{-1}}f,g\rangle \\
&= \langle \Ad_{(\Pi\exp f)^{-1}}f,g\rangle \\
& = \langle -f^{\ast-1},g\rangle.
\end{align*}
Since the map $f\mapsto -f^{\ast-1}$ is measure-preserving (Lemma \ref{l:f-1jacobian}),
\begin{align*}
&\nu_{N,R_N}\left\{f\in R_NB_N\colon \left\vert\langle f,\Ad_{\rho_N^f}g\rangle\right\vert>\norm{f}
\e\right\}\\
&= \nu_{N,R_N}\left\{f\in R_NB_N\colon \left\vert\langle -f^{\ast-1},g\rangle\right\vert>\norm{f}\e\right\} \\
&= \nu_{N,R_N}\left\{f\in R_NB_N\colon \left\vert\langle f,g\rangle\right\vert>\norm f\e\right\}\\
& <C\exp(-c^\prime \e^2 N).
\end{align*}
For sufficiently small values of $\alpha$, the function $\arccos\alpha$ is Lipschitz continuous, so we conclude that for suitable $C^\prime,c^{\prime\prime}$ and all $\e>0$ small enough,
\[\nu_{N,R_N}\left\{f\in R_NB_N\colon \left\vert\angle(f,\Ad_{\rho_N^f}g)-\frac\pi 2\right\vert>\e\right\} <C^\prime\exp(-c^{\prime\prime} \e^2 N).\]
\end{proof}

Now we are ready to provide an estimate for $(I)$.

\begin{lemma}
The total variation distance between the measure $\nu_{N,R_N}$ and its direct image under $f\mapsto f + \Ad_{\rho^f_N}g$ goes to zero as $N\to\infty$, assuming $R_N=\omega(\sqrt N)$. 
\label{l:II}
\end{lemma}

\begin{proof}
By assumption, we have $R_N/N=\omega(N^{-1/2})$. Select a sequence 
\[\e_N\in o\left(\frac{R_N}{N}\right)\cap \omega(N^{-1/2}).\]
Note for future use (Eq. (\ref{eq:future})) that this implies in particular 
\[\frac{\e_N}{R_N}N\to 0.\]
Also, 
$\e_N^2N\to\infty$, and Lemma \ref{l:angleconc} tells us that the angle that a random element $f\in R_NB_N$ forms with $\Ad_{\rho^f_N}g$ is within the margin $\pi/2\pm \e_N$ with probability approaching one as $N\to\infty$. Now, if a Euclidean ball has radius $R_N$, then the endpoint of a vector of a fixed length tangent to the sphere is contained within the distance $O(R_N^{-1})\subseteq o(N^{-1/2})\subseteq o(\e_N)$ of the surface. We conclude: with a probability asymptotically approaching one, for a random element $f\sim \nu_{N,R_N}$ the vector $f + \Ad_{\rho^f_N}g$ is contained in the ball of radius $R_N+C\e_N$ for some $C>0$ independent of $N$. The ratio of the Euclidean volume of the $N$-dimensional ball of radius $R_N+C\e_N$ to the volume of the concentric ball of radius $R_N$ is
\begin{align}
\frac{\left(R_N+ C\e_N\right)^N}{\left(R_N\right)^N}&=
\left(1+ \frac{C\e_N}{R_N}\right)^N \nonumber \\
&= \left(1+ \frac{C\e_N}{R_N}\right)^{\left(\frac{R_N}{C\e_N}\right) \left(C\frac{\e_N}{R_N}N\right)} \nonumber\\
& \approx \exp\left(C\frac{\e_N}{R_N}N\right) \label{eq:future}\\
&\to 1\mbox{ as }N\to\infty. \label{eq:twoballs}
\end{align}

Denote for simplicity by $\phi$ the map $f\mapsto f + \Ad_{\rho^f_N}g$. Let $A$ be a Borel subset of $V_N$. We will show that $\nu_{N,R_N}(\phi(A))\to \nu_{N,R_N}(A)$ as $N\to\infty$, and the convergence is uniform in all such $A$. Denote
\[D_N=\left\{
f\in R_NB_N\colon \left\vert\angle\left(f,\Ad_{\rho_N^f}g\right)
-\frac{\pi}2\right\vert<\e_N
\right\}.\]
We have $\nu_{N,R_N}(D_N)\to 1$, and so 
\begin{equation}
\nu_{N,R_N}(A\cap D_N)\to \nu_{N,R_N}(A).
\label{eq:1conv}
\end{equation}
Also, $\phi(A\cap D_N)\subseteq (R_N+C\e_N)B_N$. 
Let $\e>0$ be any fixed value.
As $\phi$ preserves Lebesgue measure (Lemma \ref{l:id+const}), and in view of Eq. (\ref{eq:twoballs}), we have for all sufficiently large $N$
\begin{align*}
\nu_{N,R_N}(A) & \overset{\e}\approx
\nu_{N,R_N}(A\cap D_N) \\ &= \frac{\lambda_N(A\cap D_N)}{\lambda_N(R_NB_N)}\\
& =\frac{\lambda_N(\phi(A\cap D_N))}{\lambda_N(R_NB_N)} \\
&= \frac{\lambda_N(R_NB_N\cap\phi(A\cap D_N))}{\lambda_N(R_NB_N)} + \frac{\lambda_N(\phi(A\cap D_N))\setminus R_NB_N)}{\lambda_N(R_NB_N)} \\
& \overset{\e}\approx \frac{\lambda_N(R_NB_N\cap\phi(A\cap D_N))}{\lambda_N(R_NB_N)} \\
&= \frac{\lambda_N(R_NB_N\cap\phi(A))}{\lambda_N(R_NB_N)} - \frac{\lambda_N(R_NB_N\cap\phi(A\setminus D_N))}{\lambda_N(R_NB_N)}
\\
& \overset{\e}\approx \nu_{N,R_N}(\phi(A)).
\end{align*}
All the rates of convergence above only depend on $N$ and $g$ and are independent of $A$, which implies the desired statement.
\end{proof}

\begin{proof}[Proof of Lemma \ref{l:main}] 
In the notation of the Lemma and the subsequent discussion, we have:
\begin{align*}
&\left\vert\int F(f)\,d\nu_{N,R_N} - \int F(f)\,d(\nu_{N,R_N}\ast g)\right\vert 
\\
& \leq \int \left\vert F(f) - F(f\ast g)\right\vert \,d\nu_{N,R_N} \\
&\leq 
\int\left\vert F(f) - F(f + \Ad_{\rho^f_N}g) \right\vert\,d\nu_{N,R_N} +
\int\left\vert F(f + \Ad_{\rho^f_N}g) - F(f\ast g) \right\vert\,d\nu_{N,R_N}.
\end{align*}
The first integral converges to zero when $N\to\infty$ uniformly in $F$ because of Lemma \ref{l:II}, where we use the assumption $R_N=\omega(\sqrt N)$. For the second integral we have, since $F$ is a $1$-Lipschitz function and using Eq. (\ref{eq:split}):
\begin{align*}
\int\left\vert F(f + \Ad_{\rho^f_N}g) - F(f\ast g) \right\vert\,d\nu_{N,R_N} 
& \leq \int\left\Vert f + \Ad_{\rho^f_N}g - f\ast g \right\Vert_2\,d\nu_{N,R_N} \\
&= \int\left\Vert\Ad_{\Pi\exp f}g-\Ad_{\rho^f_N}g\right\Vert_2\,d\nu_{N,R_N} \\
&\leq \sup_{f\in R_NB_N}\left\Vert\Ad_{\Pi\exp f}g-\Ad_{\rho_N}g\right\Vert_2,
\end{align*}
which quantity converges to zero as $N\to\infty$ (of course uniformly in $F$) by Lemma \ref{l:I}. Here the assumption $R_N=o(N)$ is needed.
\end{proof}

\section{Amenability of path groups}
\label{s:proofs}

The amenability for the based path group $H^1_0(\I,K)$ and the free path group $H^1(\I,K)$ is deduced from Lemma \ref{l:main} by totally standard tools. 

Select a sequence of positive reals $R_N=o(N)\cap \omega(\sqrt N)$.
Fix a non-principal ultrafilter $\mathcal U$ on the natural numbers and define a mean $\phi$ for every left uniformly continuous bounded function $F$ on the group $H^1_0(\I,K)$ as follows:
\[\phi(F)=\lim_{N\to{\mathcal U}}\int F(f)\,d\nu_{N,R_N}.\]
Since the sequence of values of the integrals is uniformly bounded by $\norm{F}_{\infty}$, the ultralimit is well defined. Obviously, $\phi$ is a mean. 

Let $_gF$ denote the right translation of $F$ by an element $g$:
\[{\,}_gF(f)=F(f\ast g).\]

According to Lemma \ref{l:main}, 
\begin{equation}
\phi(_gF)=\phi(F)
\label{eq:phi}
\end{equation}
for every bounded $1$-Lipschitz function $F$ and each element $g\in H^1_0(\I,K)$
whose right logarithmic derivative is a simple function. Renormalizing, we see that the same holds for every bounded Lipschitz function $F$. 

Since the bounded Lipschitz functions on a metric space are uniformly dense in the space of uniformly continuous bounded functions (see e.g. \cite{GJ}, Corollary 1) and $\phi$ has norm one and so is continuous, we deduce the right-invariance of $\phi$ on all functions $F\in\LUCB(H^1_0(\I,K))$ and all elements $g$
whose right logarithmic derivative is a simple function. The simple functions are dense in $L^2$, and it is easy to see from the definition of a left uniformly continuous function $F$ that whenever $g_n\to g$, we have $_{g_n}F\to _gF$ uniformly. This finally implies that $\phi$ is a right-invariant mean on the space of left uniformly continuous bounded functions. But this is equivalent to amenability, as we have noted in the Introduction.

The mean as defined in the Eq. (\ref{eq:phi}) is in fact left-invariant as well, according to one of the main technical results of \cite{P20} (Lemma 11). Thus, we obtain a bi-invariant mean on the space $\RUCB(H^1_0(\I,K))$ which is a part of the statement of Theorem \ref{th:bi}.

To deduce amenability of the group $H^1(\I,K)$, recall that a topological group $G$ is amenable whenever a normal subgroup $H$ and the topological quotient group $G/H$ are amenable (\cite{E}, II.1). Now note that the group of based paths $H^1_0(\I,K)$ is a normal subgroup of $H^1(\I,K)$, with the compact quotient group $K$.

In order to show that the group $H^1(\I,K)$ satisfies the conclusion of Theorem \ref{th:bi}, we need the following.

\begin{lemma}
A topological group $G$ is both amenable and skew-amenable if and only if it admits a bi-invariant mean on the space $\RUCB(G)$ (or, equivalently, $\LUCB(G)$).
\label{l:amskew}
\end{lemma}

\begin{proof}
Sufficiency is trivially true. To prove necessity, let $RM$ denote the set of all right-invariant means on $\RUCB(G)$. This is a convex set in the dual space to $\RUCB(G)$, and by assumption it is non-empty.
Equipped with the weak$^\ast$ topology, this set is compact. Indeed, if $\xi$ is a mean that is not right-invariant, there are $g\in G$ and $f\in\RUCB(G)$ with $\e=\abs{\phi(f)-\phi(_gf)}>0$. The weak$^\ast$ open neighbourhood of $f$ (in the compact space of all means) consisting of all $\psi$ with $\abs{\psi(f)-\phi(f)}<\e/2$ and $\abs{\psi(_gf)-\phi(_gf)}<\e/2$  does not contain any right-invariant means. 

Also, $RM$ is closed under the left action by $G$, given by $^g\phi(f)=\phi(^{g^{-1}}f)$, where $^gf(x)=f(g^{-1}x)$. Indeed, notice that
\[{\,}^{g^{-1}}(_hf)(x)= f(gxh)=_h\left(^{g^{-1}}f\right)(x),\]
and so,
if $\phi$ is right-invariant, then so is $^g\phi$:
\begin{align*}
{\,}^g\phi(_hf)&= \phi\left(^{g^{-1}}(_hf)\right) \\
&= \phi\left(_h\left(^{g^{-1}}f\right)\right) \\
 &= \phi\left(^{g^{-1}}f\right) \\
&={\,}^g\phi(f).
\end{align*}
The left action of $G$ on $RM$ is by affine transformations:
\[{\,}^g\left(t\phi+(1-t)\psi\right) = t{\,}^g\phi + (1-t){\,}^g\psi.\]
Since $G$ is amenable, it has the fixed point property for the affine continuous actions on convex compact sets (see e.g. \cite{E}, Th\'eor\`eme 1), and therefore there is a mean $\phi\in RM$ that is left-invariant as well.
\end{proof}

\begin{remark}
As pointed out to the author by F. Martin Schneider, a shorter way to prove the  Lemma is by taking the convolution of a left- and a right-invariant mean on $\RUCB(G)$. For the discrete groups, the technique was used in \cite{greenleaf}, Lemma 1.1.3, then it was developed in a general context of semigroups in \cite{BJM}, Section 2.2, and adapted to the means on right uniformly continuous bounded functions in \cite{ST}, see Lemma 3.2.
\end{remark}

Now we only have to recall that the group of finite energy paths $H^1(\I,K)$ is skew-amenable \cite{P20}.

\section{Amenability of loop groups}
\label{s:loop}

Here is a tool allowing to deduce the amenability of loop groups from the amenability of corresponding path groups. 

\begin{theorem}
A co-compact normal subgroup of an amenable Polish group is amenable.
\label{th:coco}
\end{theorem}

\begin{remark}
A topological subgroup of an amenable Polish group is in general not amenable \cite{dlH1}. The same applies even in a favourable situation where the left and right uniformities of a group coincide and so amenability is equivalent to skew-amenability, and moreover, this behaviour seems to be the norm \cite{CT,AST}.

An analogous result to the above Theorem \ref{th:coco} for skew-amenable groups was obtained in \cite{P20}, and considerably strengthened in \cite{PS2} (removing the assumptions of normality and polishness). We do not know if the above result still holds without assuming that the subgroup is normal (while the assumption of the group being Polish does not seem very essential, see Remark \ref{rem:spectra}).
\end{remark}

Let us recall Mackey's construction of an induced action (see e.g. Becker and Kechris \cite{BK}, Thm. 2.3.5). Let $H$ be a closed topological subgroup of a topological group $G$, and let $H$ act continuously on a topological space $X$. Assume $G$ and $X$ Hausdorff. Equip the product $X\times G$ with the natural action of $H$ on the left: $h\cdot (x,g)=(h\cdot x, hg)$. Denote the $H$-orbit of $(x,g)$ by $[x,g]$. The orbit space $(X\times G)/H$, which we will denote $X\times_HG$ and equip with the quotient topology, carries a well-defined action of $G$ given by $\gamma\cdot [x,g]=[x,g\gamma^{-1}]$. The saturation of an open subset by the $H$-orbits is open. It follows that the action of $G$ on $X\times_HG$ is continuous, and the quotient map $q\colon X\times G \to X\times_H G$ is in fact open.

Hence, the Cartesian square of $q$ is open too (\cite{engelking}, Prop. 2.3.29), in particular is a quotient map. One readily verifies that the $H$-orbit equivalence relation on $X\times G$ is closed (as a subset of the square of the latter space). It follows that the diagonal of the square of $X\times_HG$ is closed, that is, the latter space is Hausdorff (\cite{gottschalk}, p. 123, (5)).

\begin{lemma}
If $X$ is compact and $H$ is co-compact in $G$, then $X\times_HG$ is compact.
\label{l:xxhg}
\end{lemma}

\begin{proof}
The second coordinate projection $X\times G\to G$ gives rise to an open map between the quotient spaces of $H$-orbits: $\pi\colon X\times_HG\to H\backslash G$. This map sends $[x,g]$ to the coset $Hg$. Define an action of $G$ on the right coset space $H\backslash G$ by $\gamma\cdot Hg=Hg\gamma^{-1}$. Then the map $\pi$ is $G$-equivariant: $[x,g\gamma^{-1}]$ goes to $Hg\gamma^{-1}$. 

Each fibre of $\pi$ is homeomorphic to $X$, hence compact. (Enough to check for the fibre over the coset $H$, where a homeomorphism is established by $x\mapsto [x,e]$.) The continuous action of $G$ on $X\times_HG$ lifts to a continuous action of $G$ on the set of all compact subsets of $X\times_HG$ equipped with the Vietoris topology. This action, restricted to the family $\mathcal F$ of all fibres $\pi^{-1}(Hg)$, $g\in G$, is transitive. The subgroup $H<G$ fixes the fibre $\xi=\pi^{-1}(H)$. Since $G/H$ is compact, the orbit map $G\ni g\mapsto g\cdot \xi$ factors through a continuous equivariant map from $G/H$ onto the orbit of $\xi$. Consequently, the family $\mathcal F$ of all fibres of $\pi$ is compact in the Vietoris topology. Hence the union of elements of this family is compact. But this union is $X\times_HG$.
\end{proof}

\begin{remark}
The author has learned about the result from the anonymous referee of the paper and could not find a reference in this generality. Here is a justification for what appears to be a more complicated proof than necessary.

In the classical case where $G$ is a Lie group, the proof follows from the existence of a local section $s\colon V\to G$, where $V$ is open in $H\backslash G$. The composition of the map $v\mapsto (x,s(v))$ (for some fixed $x\in X$) with the quotient map $q\colon X\times G \to X\times_H G$ is a local section of $\pi$, making $X\times_HG$ a locally trivial bundle over $H\backslash G$ with a compact fibre. 

This, however, cannot be assumed in a more general situation. An example by Karube \cite{karube} shows that the fibration $X\times_HG$ is not in general locally trivial over $H\backslash G$ even if $G$ and $X$ are compact. Indeed, let $G=\T^\N$ be the infinite countable power of the circle rotation group, and let $H$ be the countable power of the subgroup of $\T$ of order two. Let $H$ act on itself, $X=H$, by left translations. The total space $X\times_HG$ is seen to be homeomorphic to $G$, thus locally path-connected, so cannot be locally trivial with a totally disconnected fibre $X=H$.

A more general approach is to use the notion of a syndetic subgroup. Recall that $H$ is syndetic in $G$ if there is a compact set $K$ with $HK=G$ (of course $HK$ can be replaced with $KH$). Now assume that $H$ is syndetic in $G$.
The restriction of $q\colon X\times G \to X\times_H G$ to the compact set $X\times K$ is onto, proving the compactness of the induced $G$-space (see Gottschalk \cite{gottschalk}, p. 123, (6)). 

Every syndetic subgroup is clearly co-compact. The converse is easily verified  for locally compact groups. Moreover, it is true for completely metrizable groups, due to Bourbaki's result (\cite{bourbaki}, \S 2, n 10, Proposition 18): if $q\colon X\to Y$ is a continuous open map from a complete metric space to a Hausdorff space, then every compact subset $K\subset Y$ is an image under $q$ of a compact subset of $X$. More generally, the result can be extended over \v Cech-complete spaces $X$, see \cite{Arh1}. Thus, the statement of Lemma \ref{l:xxhg} holds for \v Cech-complete groups, in particular, Polish groups.

However, for more general topological groups $G$ co-compactness of a subgroup $H$ does not imply $H$ being syndetic, even assuming $G$ complete in the stronger sense of Weil (in the one-sided uniformity). Here is an example.

Recall that for a Tychonoff space $X$ the free topological group on $X$, denoted $F(X)$, is just the free group in the algebraic sense equipped with the finest group topology inducing the given topology on $X$. It is Hausdorff. (See \cite{Arh}, \S 4.)
Take now as $X$ a copy of the circle rotation group, $\T$, equipped with a topology whose points are all discrete except the identity, whose filter of neighbourhoods is the same as in the usual topology of $\T$. The identity map $X\to\T$ is continuous, so it extends to a continuous surjective group homomorphism, $h$, from the free topological group on $X$ onto $\T$. Moreover, the homomorphism $h$ is a quotient map, because the standard topology on $\T$ is the finest group topology making $h$ continuous. As a consequence, the closed subgroup $H=h^{-1}(1)$ is co-compact in $G=F(X)$, because the abelian topological factor-group $G/H$ is isomorphic to $\T$. 
For every completely regular space $X$, one can represent $F(X)$ as the union of a countable family of subspaces homeomorphic to open subspaces of suitable finite powers of $X$ (\cite{Arh}, (5.1)). It follows that every compact subspace of $F(X)$ is the union of countably many subspaces homeomorphic to compact subspaces of finite powers of $X$. Each finite power $X^n$ of our space $X$ is obviously a scattered space, that is, contains no subspace that is perfect in itself. A compact subspace of the metrizable space $X^n$ is second-countable, and being scattered, is countable by a result of Cantor--Bendixson--Hausdorff (\cite{engelking}, Exercise 1.7.11; \cite{BK}, Theorem 6.4).
Hence all compact subspaces of $F(X)$ are countable. This implies $H$ is not syndetic in $F(X)$. Finally, notice that, since $X$ is metrizable, $F(X)$ is complete in the Weil sense \cite{U}.
%
\label{rem:induced}
\end{remark}

\begin{proof}[Proof of Theorem \ref{th:coco}]
Let $G$ be an amenable Polish group, and let $H\triangleleft G$ be a closed normal subgroup such that the topological group $K=G/H$ is compact. Let $H$ act continuously on a compact space $X$. We will verify that $X$ admits an invariant regular Borel probability measure. Since $H$ is Polish, it is well known and easily seen that one can assume $X$ to be metrizable. (In fact, one can take as $X$ the Hilbert cube $\I^{\aleph_0}$, \cite{AMP}, Th\'eor\`eme 3.9.) Form the induced $G$-space $Y=X\times_HG$. It is compact (by Lemma \ref{l:xxhg}) and metrizable as well. The fibres of the canonical $G$-equivariant quotient map $q\colon Y= X\times_HG\to H\backslash G$ are the translations of copies of $X$: $q^{-1}(Hg)=g^{-1}\cdot X$, where $X$ is embedded into $Y$ via $X\ni x\mapsto [x,e]$. As $H$ is normal in $G$, each fibre is $H$-invariant: for every $g\in G$, $x\in X$ and $h\in H$
\begin{equation}
h\cdot [x,g] = [x,gh^{-1}] = [x,gh^{-1}g^{-1} g]=
[ghg^{-1}x,g]\in q^{-1}(Hg).
\label{eq:haction}
\end{equation}

Since $G$ is amenable, there is a $G$-invariant Borel probability measure $\mu$ on $Y=X\times_HG$. Denote by $q_{\ast}\mu$ the direct image of $\mu$ under $q$, which is a $G$-invariant probability measure on $H\backslash G$.
The Rokhlin disintegration theorem (see e.g. Theorem 5.3.1 in \cite{AGS}) implies: there is a family of Borel probability measures $(\mu_{Hg})$ supported on $q_{\ast}\mu$-almost every fibre of $q$
with the property that, given any bounded Borel function $f$ on $Y$,
\begin{equation}
\int_Y f(x)\,d\mu(x) = \int_{H\backslash G} d q_{\ast}\mu \int_{q^{-1}(Hg)} f(y)\,d\mu_{Hg}(y).
\label{eq:rokhlin}
\end{equation}
Moreover, such a family is essentially unique. 

Choose a Borel selector $s\colon H\backslash G\to Y$, that is, a Borel measurable map with $s(Hg)\in q^{-1}(Hg)$ for all $g$ (\cite{BK}, Theorem 12.16.) For simplicity, we will denote $s(Hg)=\bar g$. 
The selector allows us to canonically identify each fibre $q^{-1}(Hg)$ with $X=q^{-1}(H)$ as compact $H$-spaces, through the translation by $\bar g$.

Fix $h\in H$. Because of invariance, the translate $h\cdot\mu$ of $\mu$ by $h$ equals $\mu$. Performing a change of variables and using Eq. (\ref{eq:haction}), we conclude that the family of measures $\bar gh\bar g^{-1}\cdot \mu_{Hg}$ on a.e. fibre $q^{-1}(Hg)$ forms another disintegration for $\mu$. The essential uniqueness of the disintegration implies that for $q_{\ast}\mu$-almost all $Hg$,  the measure $\mu_{Hg}$ is $\bar gh\bar g^{-1}$-invariant. Now let $H^{\prime}$ be a countable dense subgroup of $H$. For $q_{\ast}\mu$-almost all cosets $Hg$, the measure $\mu_{Hg}$ is invariant under the subgroup $\bar gH^{\prime}\bar g^{-1}$, dense in $H$. Fix any such $Hg$ and transfer the measure $\mu_{Hg}$ to $X=\mu_H$ by means of the translation by $\bar g$ as mentioned above. We get a Borel probability measure, $\bar g\cdot\mu_{Hg}$, on $X$ that is invariant under a dense subgroup of $H$. The action of $H$ on $X$ is continuous, and now the Stone representation theorem, identifying probability measures with states on $C(X)$, leads to conclude that $\bar g\mu_{Hg}$ is invariant under all of $H$, and we are done.
\end{proof}

\begin{remark}
One cannot repeat the proof word for word for more general topological groups. The Rokhlin decomposition theorem is valid in a great generality (Pachl \cite{pachl}, see also Fremlin (\cite{F}, Theorem 452I). Unfortunately, the results about the uniqueness of measure disintegration seem to only be known for the standard Lebesgue measure space, even if the treatise \cite{F} does not present a counter-example either.

However, we believe that this difficulty can be circumvented. Constructing $X$ and $Y$ as in the above proof, it is enough to show that every countable subgroup $H_1$ of $H$ admits a regular Borel probability measure on $X$ that is $H_1$-invariant. And with such a subgroup fixed, it appears to us that one can reduce the space $Y$ and the fibration to the metrizable case using the technique of inverse spectra (see e.g. \cite{AMP}).
\label{rem:spectra}
\end{remark}

Let us finish the proof of Theorem \ref{th:main}.
The based loop group $H^1_0(\s^1,K)$ is a normal co-compact subgroup of $H^1_0(\I,K)$ with the quotient group $K$, so it amenable by the above Theorem \ref{th:coco}. The free loop group $H^1(\s^1,K)$ contains the based loop group $H^1_0(\s^1,K)$ as a normal subgroup, with the compact quotient isomorphic to $K$, and so is amenable by \cite{E}, II.1. 

To finish the proof of Theorem \ref{th:bi}, note that
both groups $H^1_0(\s^1,K)$ and $H^1(\s^1,K)$ are also skew-amenable by \cite{P20}, so both of them admit bi-invariant means on the right uniformly continuous bounded functions, as per Lemma \ref{l:amskew}.

\section*{Acknowledgements} The author is indebted to the anonymous referee of this paper who in particular suggested using the induced action as a means to considerably simplify the original proof of Theorem \ref{th:coco}. F. Martin Schneider has made a number of useful remarks on the original version. Thanks go to Ugo Bruzzo and to Vladimir V. Uspenskij for making to the author known, among others, the results of, respectively, Karube and Bourbaki, referred to in Remark \ref{rem:induced}.

During the preparation of this article, the author was supported by the DCR-A fellowship 300050/2022-4 of the Program of Scientific and Technological Development of the State of Para\'\i ba, Brazil offered by CNPq and FAPESQ.

\end{document}